\def\anonymous{0} 

\documentclass[runningheads,orivec]{llncs} % orivec option returns \vec to standard
\usepackage{amsfonts}
\usepackage{amsmath}
\usepackage{amssymb}
\usepackage{url}
\usepackage[hidelinks]{hyperref}
\usepackage{graphicx}

\newcommand{\bx}{\boldsymbol{x}}

\newcommand{\ul}{\mathbf}
\newcommand{\R}{\mathbb{R}}

\newcommand{\cG}{\mathcal{G}}

\newcommand{\cA}{\mathcal{A}}

\newcommand{\Ad}{\text{Ad}}
\newcommand{\ad}{\text{ad}}
\DeclareMathOperator{\argmin}{argmin}
\DeclareMathOperator{\SE}{SE}
\DeclareMathOperator{\SO}{SO}

\begin{document}

\title{Analysis and Computation of Geodesic Distances on Reductive Homogeneous Spaces. % with Applications
}
\titlerunning{Analysis of Geodesic Distances on Reductive Homogeneous Spaces}

{\if\anonymous0
    \author{
        R.~Duits$^1$, 
        G.~Bellaard$^1$, 
        A.B.~Tumpach$^2$. 
    }
    \authorrunning{
        R.~Duits, 
        G.~Bellaard, 
        A.B.~Tumpach. 
    }
    \institute{
        $^1$ CASA \& EAISI, Department of Mathematics \& Computer Science, Eindhoven University of Technology, The Netherlands.\\
        $^2$ Wolfgang Pauli Institute, Vienna, Austria \\
        \email{
            \{r.duits,
            g.bellaard\}@tue.nl, 
            alice-barbora.tumpach@univ-lille.fr
        }
    }
\else
    \author{Anonymous Author(s)}
    \authorrunning{Anonymous Author(s)}
    \institute{Anonymous Institute(s)}
\fi}
    
\maketitle
\begin{abstract}
Many
geometric machine learning and image analysis applications, require a left-invariant metric on the 5D homogeneous space of 3D positions and orientations $\SE(3)/\SO(2)$. This is done in Equivariant Neural Networks (G-CNNs), or in PDE-Based Group Convolutional Neural Networks (PDE-G-CNNs), where the Riemannian metric enters in multilayer perceptrons, message passing, and max-pooling over Riemannian balls.
In PDE-G-CNNs it is proposed to take the minimum left-invariant Riemannian distance over the fiber in $\SE(3)/\SO(2)$,
whereas in G-CNNs and in many geometric image processing methods an $\SO(2)$-conjugation invariant section $\sigma$ is advocated.\\
The conjecture rises whether that 
computationally much more efficient section $\sigma$
indeed always selects distance minimizers over the fibers.
We show that this conjecture does \emph{not} hold in general, and in the logarithmic norm approximation setting used in practice we analyze the small (and sometimes vanishing) differences.
We first prove that the 
minimal distance section $\sigma_d$ is reached by minimal horizontal geodesics with constant momentum and zero acceleration along the fibers, and we generalize this result to (reductive)
homogeneous spaces with legal metrics and commutative structure groups.
\end{abstract}
\keywords{Riemannian Geometry, Sub-Riemannian Geometry, Fiber Bundles}

\section{Introduction and Background}

\subsection{Research context}

In many geometric image analysis applications one has
to include multi-orientation image processing that relies on
left-invariant Riemannian metrics and their logarithmic norm approximations.

For example, in diffusion-weighted MRI 
one obtains after some inverse methods processing an orientation density function of water molecules that are generally believed to follow the biological fibers in the brain. These inverse methods are effective, but involve false peaks that are not aligned with peaks in neighboring spherical distributions. Therefore contextual image processing group convolution with heat-kernel approximations on the 5D homogeneous space $\mathbb{M}_3=\SE(3)/\SO(2)$ of 3D positions and orientations is needed to clean the fiber structures in the data \cite{portegies2015improving,DuitsIJCV2010},  valuable in clinical applications \cite{portegies2015improving,Meesters}.

Furthermore, in many roto-translation equivariant deep learning methods one includes similarity metrics between points on $\mathbb{M}_{3}$ (`local orientations') where one again relies on left-invariant Riemannian metrics on the Lie group quotient
$\mathbb{M}_{3}$ and their logarithmic norm approximations.
For instance in \cite{bekkers2023fast}
where `fast, expressive $\SE(3)$ equivariant neural networks' (G-CNNs) are achieved with many applications. There the neural network kernels %attributes
follow symmetry considerations in geometric image processing \cite{portegies2015new}.

In these works and in %related works 
\cite{Tumpach,Tumpach2,bolelli2023neurogeometric}, a convenient choice of section in the quotient $G/H$ is helpful for carrying differential geometrical tools from the group $G$ towards $G/H$ in a computationally tangible way. For $G=\SE(3)=\R^3 \rtimes \SO(3)$ and $H=\SO(2)$ we advocate a specific section $\sigma:G/H \to G$ as we explain next. This specific section $\sigma$ minimizes the angular velocity (as we will prove later in Prop.~1) and as we will show in Theorem 3 it often approximates the computationally much more demanding section $\sigma_d$ that selects in each fiber the (sub)-Riemannian distance minimizer useful for a truly optimal alignment in position and orientation space $G/H$ relevant for all aforementioned applications.  

\subsection{Notation}

Consider $G=\SE(3)$ with product $g_1g_2=(\ul{x}_1,R_1) (\ul{x}_2,R_2)=(\ul{x}_1 + R_1 \ul{x}_2, R_1 R_2)$.
We express rotations in Euler angles $R=R_{e_z,\gamma} R_{e_y,\beta} R_{e_z,\alpha} \in \SO(3)$, with angles $\beta \in (0,\pi)$, $\alpha, \gamma \in [0,2\pi)$, and we use canonical coordinates $c^i$ for $\SE(3)$:
\begin{equation}\label{Ai}
    \begin{array}{l}
        (\ul{x},R_{e_z,\gamma} R_{e_y,\beta} R_{e_z,\alpha})=g=\exp \left( \textstyle \sum \limits_{i=1}^6 c^i A_i\right) \
        \textrm{
        w.r.t. 
        basis in } \mathfrak{g}=T_e(G):\\
        \{A_i\}_{i=1}^6\equiv \{\left.\partial_{x}\right|_{e}
        \left.\partial_{y}\right|_{e},
        \left.\partial_{z}\right|_{e},
        \left.\partial_{\tilde{\gamma}}R_{e_x,\tilde{\gamma}}
        \right|_{\tilde{\gamma}=0},
        \left.\partial_{\tilde{\beta}}R_{e_y,\tilde{\beta}}
        \right|_{\tilde{\beta}=0},
        \left.\partial_{\tilde{\alpha}}R_{e_z,\tilde{\alpha}}
        \right|_{\tilde{\alpha}=0},
        \}
    \end{array}
\end{equation}
with unit element $e=(\ul{0},I) \in G$. 
Consider subgroup $H = \{(\ul{0},R_{\ul{e}_{z},\alpha})|\; \alpha \in [0,2\pi)\}\equiv \SO(2) = \textrm{Stab}_{\SE(3)}(\ul{0},\ul{a})$, with $\ul{a}=\ul{e}_z=(0,0,1)$, the quotient  $G/H$, and canonical projection $\pi: G\rightarrow G/H$.
The vertical subbundle of $T(G)$ will be denoted by $\mathcal{V} = \ker \pi_*$. The horizontal part then follows by taking the orthogonal complement $T(G)=\mathcal{V}^{\bot} \oplus \mathcal{V} =\mathcal{H} \otimes \mathcal{V}$. 
Let \mbox{$\textrm{Ad}=(\textrm{conj})_*$} be the derivativeof conjugation at $e \in G$, i.e. $\textrm{Ad}(h)= \textrm{conj}_*(h)=(L_h)_* (R_{h^{-1}})_*$, with $R_hg=gh$ and $L_hg=hg$. 

On $G=\SE(3)$, we consider the following  metric tensor field:
\begin{equation} \label{eq:legal_metric_se3}
    \cG=g_{11} (\omega^1 \otimes \omega^1+ \omega^2 \otimes \omega^2) + g_{33}\, \omega^3 \otimes \omega^3 +
    g_{44} (\omega^4\otimes \omega^4+ \omega^5 \otimes \omega^5) + g_{66} \, \omega^{6} \otimes \omega^6
\end{equation}
where $\{\omega^{i}\}_{i=1}^6$ is the dual frame of the left-invariant frame $\{\cA_i\}_{i=1}^6$ given by: \(\left.\cA_{i}\right|_{g}=(L_g)_*A_i\),  \(\langle\omega^{i}, \cA_{j}\rangle=\delta^{i}_j\), with \(A_i\) defined by \eqref{Ai}. 

Along a geodesic $\gamma$, we write $\dot{\gamma}(t)=\sum_{i=1}^{\textrm{dim}(G)} u^{i}(t) \left.\mathcal{A}_{i}\right|_{\gamma(t)} \in T_{\gamma(t)}(G)$ for its velocity, and \(\lambda(t)=\sum_{i=1}^{\textrm{dim}(G)}\lambda_i(t) \left.\omega^{i}\right|_{\gamma(t)} \in T^{*}_{\gamma(t)}(G)\) for its momentum.  
We will denote by $d_{\cG}$ the Riemannian metric on $G$ associated to $\cG$.
We will denote by $\rho_\cG$ the logarithmic norm: $\rho_\cG(g) = \|\log g \|_\cG$, where $\log : G \to \mathfrak{g}$ is the Lie group logarithm.

\subsection{Sections $\sigma$, $\sigma_d$ and $\sigma_\rho$ of  $\pi: \SE(3) \rightarrow \SE(3)/\SO(2)$}

Both \cite{bekkers2023fast,portegies2015new} propose to take the following section $\sigma$ of $\pi: G\rightarrow G/H$:

\begin{definition}[Section $\sigma$] 
    Define $\sigma: G/H \to G$ by the condition $\alpha=-\gamma$:
    \begin{equation} \label{sigma}
        \sigma([\ul{x},R_{e_z,\gamma} R_{e_y,\beta} R_{e_z,\alpha}])=
        (\ul{x}, R_{e_z,\gamma} R_{e_y,\beta} R_{e_z,-\gamma}), \textrm{ for all } \ul{x}=(x,y,z) \in \R^3.
    \end{equation}
\end{definition}

The strength of this section is best visible in canonical coordinates $c^i$ for $\SE(3)$ since, by the computation of the logarithm in $\SE(3)$ given in
\cite[eq.55]{DuitsIJCV2010}, we have
\begin{equation} \label{eq:c6_expression}
    c^{6}
    =\sin(\alpha +\gamma) \; \cos^{2}(\beta/2) / \textrm{sinc}(q) 
    \ \ \Rightarrow \ ( c^{6}=0 \; \; \Leftrightarrow \; \; \alpha=-\gamma),
\end{equation} 
under the condition for the rotation angle \mbox{$q:=\sqrt{|c^4|^2 \!+\! |c^5|^2 \!+\!|c^6|^2} < \pi$.}
So, by \eqref{eq:c6_expression}, $\sigma$ selects in  fiber $[g]:=gH$ the element $\sigma([g]):=g_0 \in [g]$ with $c^6(g_0)=0$. \\
The equivalence in \eqref{eq:c6_expression} \emph{also} follows by $e=\sigma(H)$ and translation-invariance of $\sigma$. 

Other sections are also proposed in sub-Riemannian image processing \cite{Rodr} and stereo vision \cite{bolelli2023neurogeometric}.
The choice for section $\sigma$ defined by \eqref{sigma} is motivated by: 
\begin{enumerate}
    \item low computational costs,
    \item symmetries considerations, see Lemma~\ref{lem:ref} and \cite[Thm.1, Figs.2\&3]{portegies2015new},
    \item  vanishing torsion \cite{DuitsIJCV2010} of the spatial part of SR-horizontal exp-curves,
    \item $\sigma([e])$ is perpendicular to the direction of the fiber $[e] = H$.
\end{enumerate}

In this article, we will moreover show that $\sigma$ is close to two other sections:

\begin{definition}[Section $\sigma_d$] \label{def:section} 
    Define  $\sigma_d: G/H \to G$ by
    \begin{equation}\label{statement}
        \begin{array}{l}
            \forall_{g \in G}\;:\; 
            \sigma_d([g]) = \underset{p\,\in \,[g]=gH}{\argmin}\ d_\cG(p, e), \textrm{with Riemannian distance } \\
            d_{\cG}(p,e)=
            \inf \limits_{
            {\scriptsize
            \begin{array}{c}
            \gamma \in {\rm PC}^{1}([0,1], G), \\
            \gamma(0)=e, \gamma(1)=p
            \end{array}
            }} \int_0^{1} \sqrt{\cG_{\gamma(t)}(\dot{\gamma}(t),\dot{\gamma}(t))}\, {\rm d}t\ . 
        \end{array}
    \end{equation}
\end{definition}

The (non-necessary smooth) section $\sigma_d$ selects the true distance minimizer(s)
over the fibers $[g]=gH$ for the Riemannian distance $d_{\cG}(\cdot,e)$ towards $e$.
This section is proposed in PDE-based equivariant CNNs (PDE-G-CNNs) in~\cite{smets2022pdebased}. 

\begin{definition}[Section $\sigma_\rho$] Define  $\sigma_\rho: G/H \to G$ by
    \begin{equation}\label{sigma_rho}
    \sigma_{\rho}([g]):=\argmin_{p \in [g] } \rho_{\cG}(p) = \argmin_{p \in [g]} \|\log p\|_{\cG}.
    \end{equation}
\end{definition}
We will analyze sections $\sigma, \sigma_{\rho}, \sigma_d$. Our main results are: 
\begin{enumerate}
    \item We show in Thm.~\ref{thm:no_acceleration} using Pontryagin Max Principle \cite{agrachevbook} that (sub-)Riemannian geodesics in $(G, \cG)$ have \emph{constant momentum  and zero acceleration along the fibers}. 
    As a consequence we recover the Riemannian submersion theorem \cite{Gallot1987,Mennucci2013,kling95,Michor} stating that the minimal distance section $\sigma_d$ yields the set of points reachable from $e$ by a \emph{minimizing} Riemannian horizontal geodesic.
    
    \item We generalize Thm.~\ref{thm:no_acceleration} to  homogeneous spaces with legal metrics (Def.~\ref{def:legal_metric}) in Thm~\ref{th:3}.
    
    \item  We show in  Prop.~\ref{prop:SO3} that $\rho_{\cG}(\sigma([g])) \geq \rho_{\cG}(\sigma_{\rho}([g])) \geq d_{\cG}(\sigma_d([g]), e)$ and that sections $\sigma,\sigma_d, \sigma_{\rho}$ do coincide %when constrained to 
    on $\SO(3)/\SO(2)\equiv S^2$.
    
    \item We show in Thm.~\ref{thm:sections} that the smooth section $\sigma$ -- \emph{which is much easier to compute than $\sigma_d$ or $\sigma_\rho$} -- is often close to sections $\sigma_d$ and $\sigma_{\rho}$.
\end{enumerate}

\section{Minimal Distance Elements in the Fibers $[g]$}

\begin{definition}[Legal Metric] \label{def:legal_metric}
    A left invariant Riemannian metric $\cG$ on $G$ given by an $Ad(H)$-invariant scalar product in $\mathfrak{g}=T_e(G)$ is called a \emph{legal} metric.
\end{definition}

Legal metrics are  left-invariant metrics on $G$ which descend to Riemannian 
metrics on $G/H$. Metric $\cG$ is legal iff
$\forall_{g_1,g_2,q \in G} \forall_{h \in H}:d_{\cG}(qg_1,qg_2)=d_{\cG}(g_1h,g_2h)$. 

\begin{figure}
    \centering
    \includegraphics[width=0.32\linewidth]{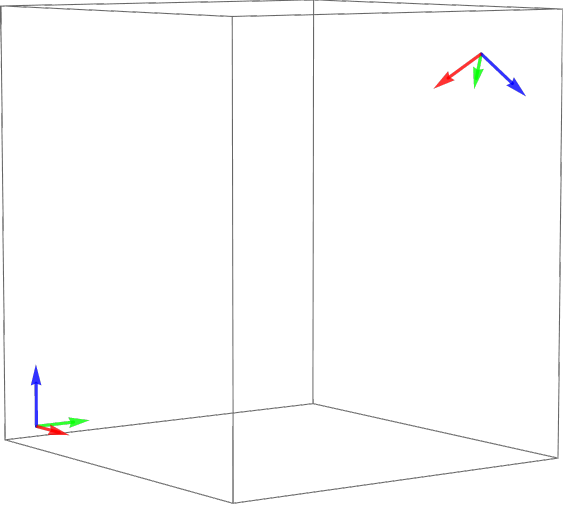}
    \includegraphics[width=0.32\linewidth]{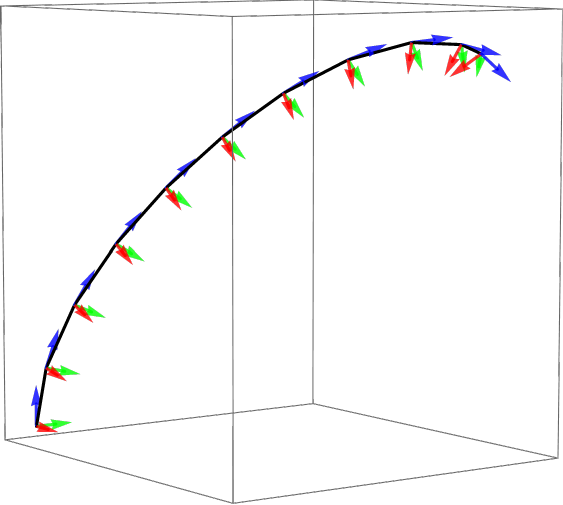}
    \includegraphics[width=0.32\linewidth]{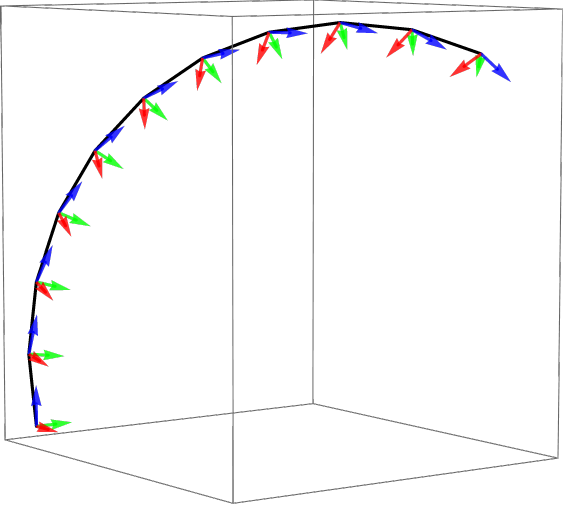}
    \caption{
        Left to right:
        \textbf{1)} start-frame $e$ and (far) end-frame $p$,
        \textbf{2)} the \emph{exact} SR geodesic \cite{DuitsJDCS} being the minimizing curve in (\ref{statement}) with $g_{11}=g_{44}=1, g_{22}=g_{33}=\infty$ connecting $e$ and $p$,
        \textbf{3)} the minimal horizontal exponential curve (ending with same orientation in blue).
        Above we visualize points $g=(\mathbf{x},R)$ in $\SE(3)$ with a local frame in $T_{\ul{x}}(\mathbb{R}^3)$.
    }\label{fig:intro}
\end{figure}

From the commutator relations in the Lie algebra of $\SE(3)$, it follows that $\cG$ defined by \eqref{eq:legal_metric_se3} is legal. 
It is a Riemannian (R) metric for $g_{11}, g_{33}, g_{44}, g_{66}$ positive, a Gauge-invariant (GI) metric \cite{Tumpach201646,Tumpach,Tumpach2} for $g_{11}, g_{33}, g_{44},$ positive  and $g_{66} = 0$, and a Sub-Riemannian (SR) metric for $g_{33}$, $g_{44}$, $g_{66}$ positive and $g_{11}=\infty$.
In the SR case, one can constrain the curves in the metric optimization to $\Delta := \langle \cA_1,\cA_2,\cA_6 \rangle^\bot$. 
This SR case is depicted in Fig.~\ref{fig:intro}.
In general $(G, \cG)$ is geodesically complete, also in the SR setting by the Chow-Rashevskii theorem.

\begin{definition}
    A left-invariant subbundle $\mathcal{H} \subset T(G)$ is called
    \begin{equation} \label{Delta}
        \begin{array}{l} 
            \textrm{Riemannian (R) horizontal if } \mathcal{H} = \mathcal{V}^{\bot}=\langle \cA_6 \rangle^{\bot}
            , \\
            \textrm{Sub-Riemannian (SR) horizontal if } \mathcal{H} = \Delta
            = \langle \cA_1,\cA_2,\cA_6 \rangle^\bot
        \end{array}
    \end{equation}
\end{definition}

We have $T(G)=\mathcal{H} \oplus \mathcal{V}$ in the Riemannian setting, and correspondingly $T(G)=\Delta \oplus \langle \mathcal{A}_1,\mathcal{A}_2,\mathcal{A}_6\rangle$ in the Sub-Riemannian setting.

\begin{theorem} \label{thm:no_acceleration} 
    Set $\ul{a}=(0,0,1)$, and $G=\SE(3)$, $H=\textrm{Stab}_{\SE(3)}(\ul{0},\ul{a})\equiv \SO(2)$.
    Let $\cG$ be a legal metric on $G$ w.r.t.  $H$, with either $g_{11}$ positive (Riemannian case) or infinite (Sub-Riemannian case). Consider a solution $t \mapsto (\gamma(t), \lambda(t))$ of the corresponding geodesic flow on $T^*G.$    
    Then along fibers we have constant momentum and no acceleration:
    \begin{equation}
        \dot{u}^6(t)=\dot{\lambda}_6(t)=
        \tfrac{d}{dt}\langle \lambda(t), \left.\mathcal{A}_6\right|_{\gamma(t)}\rangle =0.
    \end{equation}
\end{theorem}

\begin{proof}
    The Hamiltonian flow  associated to left-invariant Riemannian Hamiltonian $\mathfrak{h}_R(\lambda)=|\cG^{-1}_{\gamma}(\lambda,\lambda)|^2$ of left-invariant \emph{Riemannian} manifolds is given by:
    \begin{equation} \label{HFR}
        \dot{\nu}= \vec{\mathfrak{h}}_R\,(\nu), \textrm{ with } \nu=(\gamma, \lambda) \Leftrightarrow
        \left\{
        \begin{array}{ll}
            \dot{\gamma}= \cG^{-1} \lambda \in T_{\gamma(\cdot)}(G) & \\
            \dot{\lambda}=\textrm{coad}_{\dot{\gamma}}(\lambda)& 
        \end{array}
        \right.
    \end{equation}
    where $\vec{\mathfrak{h}}_R$ is the Hamiltonian vector field on $T^*G$ associated to the Hamiltonian $\mathfrak{h}_R$ \cite{agrachevbook}  (i.e. \(({\rm d}g \wedge \rm {d}\lambda)(\cdot, \overrightarrow{\mathfrak{h}}_R) = {\rm d}\mathfrak{h}_R$), and $\textrm{coad}_v(\lambda)(w)=
    \langle \lambda, \ad(v)(w) \rangle\). 
    Here the $\ad$ operator is obtained by conjugating the usual $\ad$ on the Lie algebra
    $T_e G$ with push-forward of the left multiplication of the inverse of the base-point $\gamma(t)$. 
    That is, $\lambda$ and $\dot{\gamma}$ are identified with their left invariant extensions and the second equation in \eqref{HFR} reduces to the Euler equation on the dual of the Lie algebra.
    
    Similarly, the Hamiltonian flow of left-invariant SR manifolds of Hamiltonian $\mathfrak{h}_{SR}(\lambda)=|\cG^{-1}_{\gamma}(P_{\Delta}^*\lambda,P_{\Delta}^*\lambda)|^2$ with dual projection $P_{\Delta}^*$ is given by:
    \begin{equation} \label{HFSR}
        \dot{\nu}= \vec{\mathfrak{h}}_{SR}\,(\nu), \textrm{ with } \nu=(\gamma, \lambda) \Leftrightarrow
        \left\{
        \begin{array}{ll}
            \dot{\gamma}= \cG^{-1} P_{\Delta}^* \lambda \in \left.\Delta\right|_{\gamma(\cdot)} & \ %\textrm{(SP)}
            \\
            \dot{\lambda}=\textrm{coad}_{\dot{\gamma}}(\lambda)& \ 
        \end{array}
        \right. 
    \end{equation} 
    \emph{Now the co-adjoint action in (\ref{HFR}), (\ref{HFSR}) gives no momentum change along fibers~(\ref{coreid}).
    Thereby geodesics cannot change their velocity in the vertical (fiber) direction:}
    \begin{equation} \label{boxed}
        \begin{array}{rl}
            \textrm{Riemannian: }&    P_{\mathcal{V}}^{*}(\dot{\lambda})=0 \textrm{ \& }\lambda =\cG \dot{\gamma}\ , \\ 
            \textrm{Sub-Riemannian: }& P_{\mathcal{V}}^{*}(\dot{\lambda})=0
            \textrm{ \& }
            P_{\Delta}^*\lambda =\cG \dot{\gamma},
        \end{array}
    \end{equation}
    so it does not matter whether one puts 0, finite, or $\infty$ costs on the fiber direction and the 3 constructions of Riemannian homogeneous spaces in \cite{Tumpach,Tumpach2} now coincide!
    
 We verify this statement explicitly in left-invariant coordinates in the setting: \\ \underline{Riemannian}:
    $\lambda=\sum \limits_{i=1}^6\lambda_i \omega^{i}$,
    $\dot{\gamma}=\sum \limits_{i=1}^6 u^{i} \left.\mathcal{A}_i\right|_{\gamma}$,
    $\mathfrak{h}_R=\sum \limits_{i=1}^6 g^{ii} \lambda_i^2$, $u^j=g^{jj}\lambda_j$,
    \mbox{(\ref{boxed}) $\Rightarrow $}
    $
    \dot{\lambda}_{6}
    =\{\mathfrak{h}_R,\lambda_6\}=\sum \limits_{k=1}^{6} \sum \limits_{j=1}^{6} c^{k}_{j6} \lambda_k(t)g^{jj}\lambda_{j}(t)= g^{11}(1-1) \lambda_{1}\lambda_2 + g^{44}(1-1)\lambda_4\lambda_5 =0
    $. 
    \\
    \underline{Sub-Riemannian}:\ $\mathfrak{h}_{SR}=
    \sum \limits_{i=3}^5 g^{ii} \lambda_i^2$,
    $\dot{\lambda}_6(t)=\{\mathfrak{h}_{SR},\lambda_6\}$, where the Poisson brackets yield $\{\mathfrak{h}_{SR},\lambda_6\}=\sum \limits_{k=1}^{6} \sum \limits_{j=3}^{5} c^{k}_{j6} \lambda_k(t)u^{j}(t) =0=g_{66} \; \dot{u}^6(t)$. \\
    In both cases $\dot{\lambda}_6=\dot{u}^6=0$, so if $u^6(0)=0$, $g_{66}$ (costs along fiber) is irrelevant. $\hfill \Box$ 
\end{proof}

\begin{remark}
By the proof above, the cost $g_{66}$ in the direction of the fibers of the canonical projection $\SE(3)\rightarrow \SE(3)/\SO(2)$ is irrelevant in the computations of geodesics both for the Riemannian case and the Sub-Riemannian case. Consequently we could as well put $g_{66} = 0$ as in the gauge-invariant case \cite{Tumpach201646}.
\end{remark}
\begin{corollary}\label{cor:1}
    (Sub)-Riemannian Geodesics in $G$ that start horizontal stay horizontal. The geodesic that realizes the minimum of the distance between $e$ and the fiber of $[g]\in G/H$ is a horizontal geodesic. 
    Consequently, $\sigma_d$ selects the point(s) in the fiber that can be connected with a minimizing horizontal geodesic. 
\end{corollary}

Cor.1 recovers Riemannian submersion theory \cite{Gallot1987,Mennucci2013,kling95,Michor}, and with Thm.1 we prove constant momentum and zero acceleration along fibers. 
It also covers the SR case, aligning with \cite[Theorem~2]{duits2018optimal}, where
SR geodesics/balls are limits of Riemannian geodesics/balls when anisotropy tends to infinity.  

Next we generalize Thm.\ref{thm:no_acceleration}, Cor.~\ref{cor:1} to (reductive) homogeneous spaces \cite{arvanitogeorgos2003introduction} with legal metrics.

\begin{definition}[Reductive Homogeneous Space]
    \label{def:reductive}
    Let $G$ be a Lie group with Lie algebra $\mathfrak{g}=T_e(G)$, and $H$ a closed subgroup with Lie algebra $\mathfrak{h}:=T_e(H)$.
    A homogeneous space $G/H$ is called \emph{reductive} if there exist a subspace $\mathfrak{m} \subset \mathfrak{g}$ such that $\mathfrak{g} = \mathfrak{h} \oplus \mathfrak{m}$ and $\Ad(H)$ maps $\mathfrak{m}$ into itself.
\end{definition}

\begin{lemma}
    Let $\cG$ be a legal metric on $G$ w.r.t. Lie sub-group $H$.
    Then $G/H$ is reductive with $\mathfrak{g}=\mathfrak{h} \oplus \mathfrak{m}$, $\mathfrak{m}=\mathfrak{h}^{\bot}$. 
\end{lemma}

\begin{proof}
    It suffices to show that \([X,Y] \in \mathfrak{m} \) for all \(X \in \mathfrak{h}\) and \(Y \in \mathfrak{m}\).
    Given that \(\mathfrak{m} = \mathfrak{h}^\bot\), this is equivalent to showing that \(\cG([X,Y],Z) = 0\) for all \(Z \in \mathfrak{h}\).
    By $\Ad_H$ invariance of $\cG$ one has
    $\cG(\ad_X Y, Z) + \cG(Y, \ad_X Z) = 0 \ \forall\ Y, Z \in \mathfrak{g}, X \in \mathfrak{h}$ as $\textrm{Ad}_*=\textrm{ad}$. 
    So, 
    $\cG([X,Y],Z) = - \cG(Y,[X,Z]) = 0$ as $[X,Z] \in \mathfrak{h}$ and $Y \in \mathfrak{h}^{\bot}. \Box$
\end{proof}

\begin{theorem} \label{th:3}
    Let $\cG$ be a legal metric on $G$ w.r.t. Lie sub-group $H$. 
    The minimal  distance section $\sigma_d$ (Def.~\ref{def:section}) takes in a fiber $[g] \in G/H$ the element $g^* \in G$ reachable by a horizontal minimizing geodesic departing from $e$. 
    Moreover, if $H$ is commutative, geodesics in $G$ have constant vertical momentum and no acceleration along the fibers.
\end{theorem}

\begin{proof}
    By Riemannian submersion theory  \cite[Prop.~2.109]{Gallot1987}, \cite[Cor. 11.25]{Mennucci2013}, \cite[Cor. 1.11.11]{kling95} or \cite[cor. 26.12]{Michor}), geodesics that start horizontal remain horizontal.
    Thereby the minimizing geodesic between $[e]$ and $[g]$  in $G/H$ has the same length as its  horizontal lift starting at $e$. 
    All curves in $G$ that project to the minimizing geodesic in $G/H$ have the same horizontal component of their velocity. 
    By the Pythagorean theorem on $T(G)=\mathcal{H}\oplus \mathcal{V}$ with $\mathcal{H} = \mathcal{V}^\bot$, the shortest one is the horizontal one, with end-point $g^* = \sigma_d([g])$, and $d_{\cG}(\sigma_d(g),e)= d([g],[e])$.
    
    Now Eq. (\ref{HFR}), (\ref{HFSR}) hold generally for left-invariant (S)R problems on Lie groups: the left-invariant HF on Lie group $G$ with left-invariant metric yields the Euler equation on $T^*_e(G)$ \cite[Sec.~I.4]{Arnold_Khesin}. 
    Now $H$ commutative $\Rightarrow [\mathfrak{h},\mathfrak{h}]=0$ and
    \begin{equation} \label{coreid}
        \boxed{
        \begin{array}{l}
            \cG \textrm{ legal} \overset{\textrm{Lem.1}, [\mathfrak{h},\mathfrak{h}]=0}{\Rightarrow} P_{\mathfrak{h}} \ad(\mathfrak{h})=0 \overset{\cG \textrm{ legal }}{\Rightarrow}
            P_{\mathcal{V}} \ad(\mathcal{V})=0 \Rightarrow  \\
            P_{\mathcal{V}^*} \textrm{coad}_{\dot{\gamma}}( \lambda) \overset{(\ref{HFR}) }{=} P_{\mathcal{V}^*}
            (\dot{\lambda})=0 
            \textrm{ and }\dot{\gamma} \overset{(\ref{HFR}) }{=} \cG^{-1}\lambda
        \end{array}
        }
    \end{equation}
    where $t \mapsto (\gamma(t),\lambda(t))$ solves the geodesic flow (\ref{HFR}) on $T^{*}G$ and with \emph{orthogonal} projections: \mbox{$P_{\mathfrak{h}}: \mathfrak{g} \to \mathfrak{g}$ on $\mathfrak{h}=T_e(H)$ and $P_{\mathcal{V}}: T(G) \to T(G)$ on $\mathcal{V}$.} 
    For  vanishing acceleration along the fibers: $P_{\mathcal{V}} \ddot{\gamma}=\cG^{-1}P_{\mathcal{V}^*} \cG (  \ddot{\gamma})=\cG^{-1}P_{\mathcal{V}^*} (\frac{d}{dt}( \cG \dot{\gamma})) \overset{(\ref{coreid})}{=}0$, with  $\ddot{\gamma}:=\nabla^{[0]}_{\dot{\gamma}}\dot{\gamma}$, with $\nabla^{[0]}$ the metric compatible Lie-Cartan connection \cite{duits2021Springer}.
    ~$\Box$ 
\end{proof}

\section{Section $\sigma([g])$ is locally close to $\sigma_{\rho}([g])$ and $\sigma_d([g])$}

Henceforth we consider the Riemannian setting (with finite anisotropy), where 
one has $\sigma_\rho([g]) \approx \sigma_{d}([g])$ if $g \approx e$ as
$|d_{\cG}(\cdot,e)|^2=|\rho_{\cG}|^2(1+ O(\rho_{\cG}^2))$, cf.~\cite{bellaard2022analysis,smets2022pdebased}. 

\begin{proposition} \label{prop:SO3}
    We have
    $\rho_{\cG}(\sigma([g])) \geq \rho_{\cG}(\sigma_{\rho}([g])) \geq d_{\cG}(\sigma_d([g]), e)$
    with equality when restricting to $\{0\} \times \SO(3)/H \equiv \SO(3)/\SO(2)\equiv S^2$.
\end{proposition}

\begin{proof}
    The inequality follows from definitions $\sigma$, $\sigma_{\rho}$, $\sigma_d$ as $g_0 \in [g]$ with $\alpha=-\gamma$ is a specific point, and a connecting exp curve is a specific connecting curve.
    By Thm.~\ref{th:3} a minimal geodesic from $e$ is horizontal, so the minimizer in $[g]$ is independent of $g_{66}$, Thm.\ref{thm:no_acceleration}. 
    Choose $g_{66}=g_{44}$ for a bi-invariant metric where geodesics are exp curves, so $\rho_{\cG}(\sigma(g))\!=\!\rho_{\cG}(\sigma_{\rho}([g])) \!=\! d_{\cG}(\sigma_d([g]), e), \forall g=(\ul{0},R). \hfill \Box$
\end{proof}
\begin{remark}
By Prop.1 (the case of equality) and translation invariance of $\sigma$ (recall (\ref{sigma}))
 section $\sigma$ selects the element $g=\sigma([g])\in SE(3)$ within the fiber $[g]$ with minimal angular velocity, i.e.: 
$
\sigma([g])=
\argmin_{p \in [g]} \|P_{T(SO(3))}\log p\|_{\mathcal{G}}
$.
\end{remark}
\begin{lemma}\label{lem:ref}
    Let $\cG$ be a legal metric given by \eqref{eq:legal_metric_se3}. 
    Set $g_0=\sigma([g])$. 
    Let $[g]=(\ul{x},\ul{n}) \in G/H$ and $[e]=(\ul{0},\ul{a})\in G/H$ be co-planar: $\det(\ul{x}|\ul{n}|\ul{a})=0$. 
    Then $\rho_{\cG}(g_0 h)= \rho_{\cG}(g_0 h^{-1})$ for all $g \in G, h \in H.$
\end{lemma}

\begin{proof} 
    Let $g \in G, h \in H$. Let $[g]$ and $[e]$ be in plane $V$. Let $S_V$ be the orthogonal reflection w.r.t. $V$.
    By co-planarity we can find a torsion free connecting exp curve : select $g_0=\exp (\sum_{i=1}^5 c^i A_i)$ from $[g]$ with spatial velocity $\ul{c}^{(1)}=(c^1,c^2,c^3)$ perpendicular to angular velocity $\ul{c}^{(2)}=(c^4,c^5,0)$. Set $s=(\ul{0},S_V) \in E(3)$. 
    Then
    \begin{equation} \label{conj}
        s g_0 s^{-1}=g_0\textrm{ and }s h s^{-1}=(\ul{0}, S_V R_{\ul{e}_z,\alpha} S_V^{-1})=(\ul{0}, R_{\ul{e}_z,-\alpha})=h^{-1}.
    \end{equation}
    Now $\rho_{\cG}(g_0 h)=\|\log g_0 h\|=\| \log s(g_0 h)s^{-1}\|$ as on Lie algebra-level conjugation with reflection $s$ just flips the signs in $A_1$ and $A_4, A_5$ direction and as $\cG$, Eq.~\!(\ref{eq:legal_metric_se3}), is isotropic diagonal in planes $\langle A_1,A_2\rangle$ and $\langle A_4,A_5\rangle$. 
    Thereby we find:
    \begin{equation}
    \begin{split}
        \rho_{\cG}(g_0 h) 
        &= \|\log (g_0 h)\| 
        = \|\log (s g_0 s^{-1}) (s h s^{-1})\| \\
        &\overset{(\ref{conj})}{=} \|\log(g_0 h^{-1})\|
        = \rho_{\cG}(g_0 h^{-1}).
    \qquad \qquad \Box\end{split}
\end{equation}
\end{proof}
%The symmetry over the fibers in Lemma~\ref{lem:ref} helps us 
%to analyze the small error (depicted in Fig.~\ref{fig:counter}) 
%in the theorem below.
\begin{theorem} \label{thm:sections} 
    In general $\sigma([g]) \neq \sigma_\rho([g])$ and 
    $\rho_{\cG}(\sigma([g])) =\rho_{\cG}(\sigma_{\rho}([g])) + \textrm{Error}_{\cG}([g]),$ with $\textrm{Error}_{\cG}([g])\to 0$ if $[g] \to [e]$.
    If $g_{11}=g_{22} \geq g_{33}$ then
    \begin{equation} \label{eq:error0}
        \begin{array}{l}
            1) \;\forall_{R \in \SO(3)} \forall_{z \in \R}\;:\; \textrm{Error}_{\cG}([(0,0,z,R])=0,
            \\[6pt] 2) \;
            [g] \textrm{ co-planar and close to }[e]  
            \Rightarrow \textrm{Error}_{\cG}([g])=0.
        \end{array}
    \end{equation}
\end{theorem}

\begin{proof}
    The first statement follows from the fact that the logarithmic norm converges to the exact Riemannian distance \cite{bellaard2022analysis} and Thm~\ref{thm:no_acceleration}.
    
    Regarding item 1) in (\ref{eq:error0}):
    By Prop.~\ref{prop:SO3} it holds if $z=0$. 
    We have for spatial velocity $\ul{c}^{(1)}=(c^1,c^2,c^3)^T$ and angular velocity
    $\ul{c}^{(2)}=(c^4,c^5,c^6)^T$ 
    of $\log_{\SE(3)}g$:
    \begin{equation} \label{c1}
        \begin{array}{l}
            \textrm{spatial velocity:} \\
            \ul{c}^{(1)}(g)= \bx - \frac{1}{2} \ul{c}^{(2)}(g) \times \bx +
            f(q)\, \ul{c}^{(2)}(g) \times (\ul{c}^{(2)}(g) \times \bx), \\
            \textrm{angular velocity :} \\
            \ul{c}^{(2)}(g)=(c^4(g),c^5(g),c^6(g))=\log_{\SE(3)}(\ul{0},R)=(\ul{0},\log_{\SO(3)}R), 
        \end{array}
    \end{equation}
    $\textrm{for }g=(\ul{x},R)$,
    with $f(q)=
    (1-\frac{q}{2}\cot\left(\frac{q}{2}\right))/q^2 \geq 0$ with $f'(q)>0$.
    So \(|c^{3}(0,0,z,R)|^2\) \(= |z|^2\big|1+ f(q)(|c^{6}(0,0,z,R)|^2-q^2)\big|^2\) is minimal if $c^{6}=0$ as $1-q^2 f(q)\geq 0$, $f\geq0$.  
    Finally by assumption $g_{11}=g_{22}\geq g_{33}$ moving along $c^1$ or $c^2$ is more expensive then moving along $c^3$, and for $g=(0,0,z,R)$ we have
    $\rho_{\cG}(\sigma[g]))=\rho_{\cG}(\sigma_{\rho}[g]))$.
    
    Regarding item 2) in (\ref{eq:error0}):
    By the symmetry of Lemma~\ref{lem:ref} and the smoothness of the exp map  $h=(\ul{x},R_{\ul{e}_{z},\gamma} R_{\ul{e}_{y},\beta} R_{\ul{e}_{z},\alpha=-\gamma})$ must be a \emph{stationary point} of optimization problem $\min_{h \in H} \rho_{\cG}(gh)$, cf.Fig.~\ref{fig:counter}. 
    Now if the stationary point is a global minimum then $\textrm{Error}_{\cG}([g])$ vanishes. 
    Thereby we check the sign of $E_{\mathbf{x}}''$ of
    $
    [-\pi,\pi) \ni \alpha \overset{E_{\mathbf{x}}}{\mapsto} \rho_{\cG}(\mathbf{x},
    R_{\ul{e}_{z},-\alpha} R_{\ul{e}_{y},\beta} R_{\ul{e}_{z},\alpha}) \geq 0
    $.
    By Prop.~\ref{prop:SO3} we have $E''_{\ul{0}}(\cdot)>0$. 
    By continuity of $\mathbf{x} \mapsto E_{\mathbf{x}}''$, and $\rho_{\cG}(\ul{x},R) \to \rho_{\cG}(\ul{0},R)$ the stationary point $\alpha=-\gamma$ remains minimal for $\ul{x} \to \ul{0}$. 
    For $gh \to e$ one has $\rho_{\cG}(gh) \to d_{\cG}(gh,e)$ and by Thm~\ref{thm:no_acceleration}, the stationary-point remains
    the minimizer for $g$ close enough to $e$.
    $\hfill \Box$
\end{proof}
\begin{remark}
    On $\SO(3)/\SO(2)$ we had $\sigma=\sigma_{\rho}=\sigma_d$, cf.~Prop~\ref{prop:SO3}. On $\SE(3)/\SO(2)$ this no longer holds, though for many co-planar end-conditions it does hold.
    Numerics shows that errors are  small for the practically relevant cases, cf.~Fig.~\ref{fig:counter}.
    When $\frac{\min\{g_{11},g_{33}\}}{\max\{g_{44},g_{66}\}}\frac{\|\ul{x}\|}{\|\log (\ul{0},R)\|_{\mathcal{I}}} \gg 1$, the stationary point $\alpha=-\gamma$ for co-planar boundary conditions  (Lemma~\ref{lem:ref}) can switch from a min to a max (cf.Fig.\ref{fig:counter} top) yielding an $Error_{\cG}([g])>0$. 
    
    In future work we will compute this turning point, and analyze the error also for the no-coplanar case, using screw-motion coordinates \cite{BellaardSmets}. 
    Then Taylor expansion yields a practical refinement of section $\sigma$. 
\end{remark}
\begin{figure}
    \center
    \includegraphics[width=\hsize]{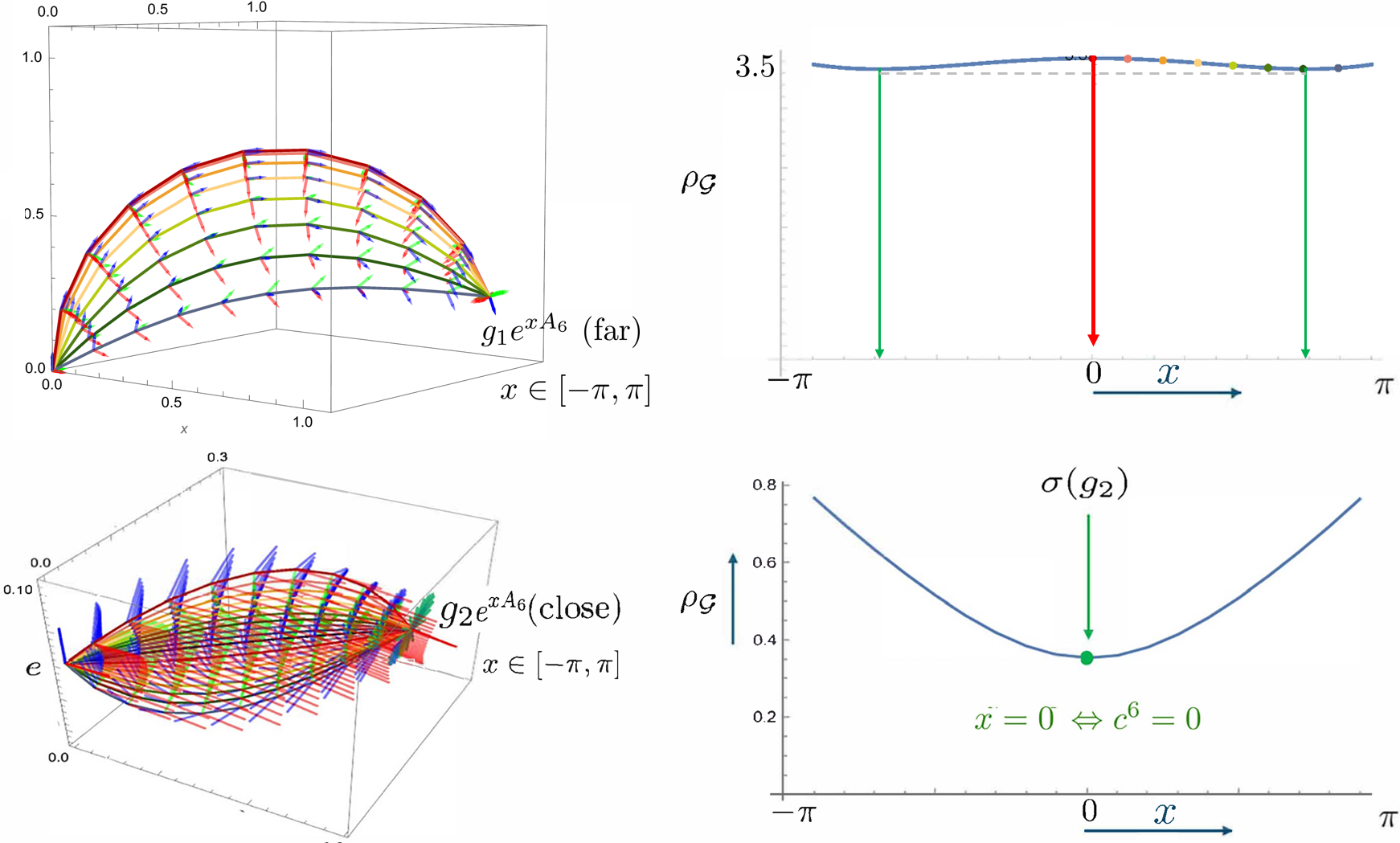}
    \caption{
        Illustration of Thm.\ref{thm:sections} and symmetry of Lemma~\ref{lem:ref}. Left: Exp curves in $\SE(3)$ that map unit element $[e]$ to a local orientation $[g_1]$ are plotted by their spatial projections along with a rotation frame.
        Right: section $\sigma$ can deviate from section $\sigma_{\rho}$ in $[g_1]$: $\textrm{Error}_{\cG}(g_1)= 0.1$.
        Settings top: $[g_1]=\{g_1 e^{x A_6}\;|\; x \in (-\pi,\pi]\}$, $g_1=\exp(2 A_3 + \tfrac{7 \pi}{16} A_4 + \tfrac{7 \pi}{16} A_5)$, co-planar, $\cG_e = \textrm{diag}(1, 1, 1, 1, 1, 0)$. 
        Here the discrepancy of the log distance between taking section $g_1=\sigma([g_1]) \Leftrightarrow x=\alpha+\gamma=0$, and the actual minimizer $\sigma_{\rho}([g_1])$ over the fiber is visible. 
        Settings bottom: For $g_2=\sigma([g_2])=\exp(\frac{1}{4}(A_3+A_2 + \frac{\pi}{14}A_5))$ (close to $e$) and $\cG_{e} = \textrm{diag}(1, 1, 1, 0.01, 0.01, 0.05)$ the error vanishes. 
    }
    \label{fig:counter}
\end{figure}
\section{Conclusion}

(Sub)-Riemannian geodesics in Lie groups $G$ with legal metrics w.r.t. commutative subgroup $H$ have constant momentum and zero acceleration along the left cosets (Thm.1\&2). 
As the minimal distance section $\sigma_d$ is reached by horizontal minimizing geodesics, we studied the effect of replacing the Riemannian distance $d_{\cG}(\sigma_d([g]),e)$ with its logarithmic norm approximation $\rho_{\cG}(\sigma_\rho[g])$. 
For $G=\SE(3)$, $H=\SO(2)$ an advocated symmetric section $\sigma$ (Lem.~\ref{lem:ref}) that coincides with $\sigma_{\rho}$ on $S^2$ (Prop.1) approximates $\sigma_{\rho}$ and $\sigma_d$ well (Thm.~\ref{thm:sections}, Fig.~\ref{fig:counter}). 
The analysis of the error and the stationary point (Fig.~\ref{fig:counter}) is left for future work.

\begin{credits}
\subsubsection{\ackname} We gratefully acknowledge the Dutch Foundation for Science NWO for funding VICI 2020 Exact Sciences (\mbox{VI.C.~202-031}).
The European Commission is gratefully acknowledged for financial support through Horizon Europe, MSCA-SE project 101131557 (REMODEL).
The authors acknowledge the excellent working conditions and interactions at Erwin Schrödinger International Institute for Mathematics and Physics, Vienna, during the thematic programme "Infinite-dimensional Geometry: Theory and Applications" where part of this work was completed.
The third Author is supported by FWF Grants I-5015 N and PAT1179524.

      \subsubsection{\discintname}
The authors have no competing interests to declare that are
relevant to the content of this article. 
\end{credits}

\bibliographystyle{splncs04}
\bibliography{literature}

\end{document}